\documentclass[12pt,a4paper,reqno]{amsart}

\usepackage{amssymb}
\usepackage{enumerate}

\newtheorem{theorem}{Theorem}
\newtheorem{proposition}{Proposition}[section]
\newtheorem{lemma}[proposition]{Lemma}
\newtheorem{corollary}[proposition]{Corollary}

\theoremstyle{definition}
\newtheorem{definition}[proposition]{Definition}
\newtheorem{remark}[proposition]{Remark}
\newtheorem{notation}[proposition]{Notation}
\newtheorem{example}[proposition]{Example}

\newcommand\dashmapsto{\mapstochar\dashrightarrow}
\newcommand{\PP}{\mathbb{P}}
\newcommand{\p}{\mathbb{P}}
\newcommand{\Z}{\mathbb{Z}}
\renewcommand{\k}{\mathbf{k}}
\newcommand{\Aut}{\mathrm{Aut}}
\newcommand{\PGL}{\mathrm{PGL}}
\newcommand{\A}{\mathbb{A}}

\newcommand{\vpb}{\vphantom{\Big)}}

\DeclareMathOperator{\Bir}{Bir}

\title{On degenerations of plane Cremona transformations}

\author{J\'er\'emy Blanc}
\author{Alberto Calabri}

\address{J\'er\'emy Blanc\\
Departement Mathematik und Informatik\\
Universit\"at Basel\\
 Spiegelgasse 1\\
  4051 Basel, Switzerland}
\email{Jeremy.Blanc@unibas.ch}

\address{Alberto Calabri\\
 Dipartimento di Matematica e Informatica\\
  Universit\`a di Ferrara\\
   Via Machiavelli 30\\
    44121 Ferrara, Italia}
\email{alberto.calabri@unife.it}

\thanks{The authors acknowledge support by the Swiss National Science Foundation Grant  "Birational Geometry" PP00P2\_128422 /1.
The second author is member of G.N.S.A.G.A.\ group of Italian I.N.d.A.M. and warmly thanks Massimiliano Mella for useful discussions.}

\begin{document}
\begin{abstract}This article studies the possible degenerations of Cremona transformations of the plane of some degree into maps of smaller degree.\end{abstract}
\maketitle

\section{Introduction}
Let us fix $\k$ to be the ground field, which will be algebraically closed of characteristic~$0$. The Cremona group $\Bir(\p^2)$ is the group of birational transformations of the plane.

There is a natural Zariski topology on $\Bir(\p^2)$, introduced in \cite{Se} (see $\S\ref{SubSec:TopologyCremona}$) and studied then in many recent texts: \cite{Bl}, \cite{Popov}, \cite{BF}, \cite{PanRit}, \cite{BCM}, \cite{Can13}. For each integer $d$, the subset $\Bir(\p^2)_d$ of elements of $\Bir(\p^2)$ of degree $d$ is locally closed, and has a natural structure of algebraic variety, compatible with the Zariski topology of $\Bir(\p^2)$. However, neither the group $\Bir(\p^2)$ nor the subset $\Bir(\p^2)_{\le d}$ of maps of degree at most $d$ (for $d\ge 2$) have a structure of an (ind)-algebraic variety \cite[Proposition~3.4]{BF}, and the bad structure comes from the degeneration of maps of degree $d$ into maps of smaller degree.

The aim of this article consists exactly in trying to understand this degeneration, and more precisely to describe the closure $\overline{\Bir(\p^2)_d}$ of $\Bir(\p^2)_d$ in $\Bir(\p^2)$, which is a subset of $\Bir(\p^2)_{\le d}$. In particular, the first question one can ask is to understand for which $d$ we have an equality $\overline{\Bir(\p^2)_d}=\Bir(\p^2)_{\le d}$ (already asked in \cite[Remark 2.16]{BF}). The complete answer is the following:

\begin{theorem}\label{Thm:ClosureBirD}
Let $d\ge 1$ be an integer. Then $\overline{\Bir(\p^2)_d}=\Bir(\p^2)_{\le d}$ if and only if $d\le 8$ or $d\in \{10,12\}$.
\end{theorem}

This theorem shows that $\Bir(\p^2)_{8}$ is not contained in the closure of $\Bir(\p^2)_9$, but is in the closure of $\Bir(\p^2)_{10}$, and the same holds replacing $8,9,10$ with $10,11,12$. One can then ask for a birational map of degree $d$ what is the minimum $k$ needed such that it belongs to the closure of $\Bir(\p^2)_{d+k}$. As we will show, there is an upper bound for $k$, depending on $d$, but no universal bound:

\begin{theorem}\label{Theorem:NoBoundOnK}
For each integer $k\ge 1$ there exists an integer $d$ and a birational map $\varphi$ of degree $d$ such that $\varphi$ does not belong to $\overline{\Bir(\p^2)_{d+i}}$ for each $i$ with $1\le i\le k$.

Every birational map of degree $d\ge 1$ is contained in $\overline{\Bir(\p^2)_{d+i}}$ for some $i$ with $1\le i\le \max\{1,\frac{d}{3}\}$.
\end{theorem}

The two theorems are obtained by a detailed study of the possible degenerations of birational maps and of the relation with their base-points. For example, we give a criterion that determines whether a birational map $\varphi$ of degree $d$ with only proper base-points, no three of them collinear, belongs to $\overline{\Bir(\p^2)_{d+1}}$: it is the case if and only if $\varphi$ has multiplicity $m_1$ and $m_2$ at two points of $\p^2$ such that $m_1+m_2=d-1$ (Corollary~\ref{Coro:GeneralLimitdplus1}). We also give three propositions that provide existence of degenerations associated to the base-points of a birational map (Propositions~\ref{prop:mi+mj=d-k}, \ref{prop:5points} and \ref{prop:5points3collinear}).

\bigskip

Let us finish this introduction by describing the situation for the subgroup $\Aut(\A^2)\subset \Bir(\p^2)$ consisting of automorphisms of the affine plane. The question of degeneration was already studied in this case, see for example \cite{Fur97}, \cite{EF04}, \cite{Fur13}. Because of Jung's theorem, every element  $f\in\Aut(\A^2)$ has a multidegree $(d_1,\dots,d_k)$ which satisfies $\deg(f)=\prod_{i=1}^k d_i$, and its length is defined to be the integer $k$. By \cite{Fur02}, the length of an automorphism is lower semicontinuous. In particular, elements of multidegree $(2,2)$ are not in the closure of $\Aut(\A^2)_{5}$, even if they are in the closure of $\Bir(\p^2)_5$, and one can construct many of such examples, using the rigidity of $\Aut(\A^2)$. The closure of the subvarieties of some given multidegrees are however not well understood, and quite hard to describe. See \cite{Fur13} for some descriptions and conjectures in the cases of length $\le 2$
and \cite{EL13} for a recent work in case of length $3$.


\section{Degree, multiplicities of base-points and homaloidal types}
\begin{definition}
Let $$\varphi\colon [x:y:z]\dashmapsto [f_0(x,y,z):f_1(x,y,z):f_2(x,y,z)]$$ be a birational map of $\p^2$, where $f_0,f_1,f_2$ are homogeneous polynomials of degree $d$ without common factor (of degree $\ge 1$).
 We will say that the \emph{degree of $\varphi$ is $d$} and that the \emph{homaloidal type} of $\varphi$ is $$(d;m_1,m_2,\ldots,m_r)$$ if the linear system of $\varphi$, which is given by the set of curves  of equation $$\sum_{i=0}^2 \lambda_i f_i(x,y,z)=0,$$ for $\lambda_0,\lambda_1,\lambda_2\in \k$, has base-points $p_1,\dots,p_r$ of multiplicity $m_1,\dots,m_r$. (Here by base-points we include all such, including infinitely near base-points).
 \end{definition}
 \begin{remark}
If $(d;m_1,m_2,\ldots,m_r)$ is the homaloidal type of a birational map of $\p^2$, the integers $m_i$ and $d$ satisfy the following equations
\begin{align} \label{Neq}
& \sum\limits_{i=1}^r m_i=3(d-1),
&
&\sum\limits_{i=1}^r (m_i)^2=d^2-1,
\end{align}
that are the classical \emph{Noether equalities}, and directly follow from the fact that the map is birational (see for instance \cite[\S2.5, page 51]{alberich}).
\end{remark}
We will use the following notation of \cite[Definition 5.2.1, page 130]{alberich}:
 \begin{definition}
 Let $d,m_1,\dots,m_r$ be integers. We will say that $T=(d;m_1,m_2,\ldots,m_r)$ is a \emph{homaloidal type of degree $d$} if it satisfies the Noether equalities (\ref{Neq}). 
 
 If there exists moreover a birational map $\varphi\in \Bir(\p^2)$ of homaloidal type $T$, then we say that $T$ is \emph{proper}, and otherwise we say that $T$ is \emph{improper}.
\end{definition}

Note that the type $(3;1,1,1,1,1,1,1,-1)$ is improper, as it contains a negative integer. Similarly, $(5;3,3,1,1,1,1,1,1)$ is another improper type, as the linear system associated to such a type would be reducible (the line through the two points of multiplicity $3$ would be a fixed component). 

In order to decide whether a homaloidal type is proper or improper, there is what is called the Hudson's test (see \cite[Definition 5.2.15, page 134]{alberich} and \cite[Definition 25 and the appendix]{BCM}). We will explain why this algorithm works, using a modern language, which is a simplified version of the action of $\Bir(\p^2)$ on a hyperbolic space of infinite dimension given by the Picard-Manin space (the interest reader can have a look at \cite[Section~3]{Can11} or \cite[Section~5]{BC}).
\subsection{Hudson's test}\label{SubSecHud}
Let us consider the free $\Z$-module $V$ of infinite countable rank, whose basis is $\{e_i\}_{i\in \mathbb{N}}$. Each homaloidal type $(d;m_1,\dots,m_r)$ corresponds to the element $de_0-\sum_{i=1}^r m_i e_i \in V$. We then consider the scalar product on $V$ given by $(e_0)^2=1$, $(e_i)^2=-1$ for $i\ge 1$ and $e_i\cdot e_j=0$ for $i\not=j$. This corresponds to the intersection of divisors on the blow-ups of $\p^2$ associated to the base-points of the corresponding maps.

We denote by $\sigma_0$ the automorphism of $V$ given by the reflection by the root $e_0-e_1-e_2-e_3$:
$$\begin{array}{llll}
\sigma_0(e_0)=2e_0-e_1-e_2-e_3,& \sigma_0(e_i)=e_i\mbox{ for }i\ge 4,\\
\sigma_0(e_1)=e_0-e_2-e_3,& \sigma_0(e_2)=e_0-e_1-e_3,& \sigma_0(e_3)=e_0-e_1-e_2.\end{array}$$
This corresponds exactly to the action of the standard quadratic map $$\sigma\colon [x:y:z]\dashmapsto [yz:xz:xy]$$
on the blow-up of the three points $[1:0:0]$, $[0:1:0]$ and $[0:0:1]$.

We then define $W$ as the group of automorphisms of $V$ generated by $\sigma_0$ and by the permutations of the $e_i$ fixing $e_0$. A simple calculation shows that $W$ preserve the intersection form and the canonical form given by $de_0-\sum m_ie_i\to 3d-\sum m_i$. Hence, the group $W$ preserves the set of homaloidal types. We then have the following:

\begin{proposition}\label{Prop:HudsonW}Let $T=(d;m_1,m_2,\ldots,m_r)$ be a homaloidal type. 

$(i)$ The type $T$ is proper if and only if it belongs to the orbit $W(e_0)$.

$(ii)$ If $T$ is proper, there is a dense open subset in $U\subset (\p^2)^r$ such that for each $(p_1,\dots,p_r)\in U$, there exists a birational transformation $\varphi$ having degree $d$ and which base-points are the $p_i$ with multiplicity $m_i$.
\end{proposition}
\begin{proof}
Suppose first that $T$ is proper, which corresponds to saying that there exists a map $\varphi\in \Bir(\p^2)$ of type $T$. By Noether-Castelnuovo theorem, $\varphi$ can be written as $$\varphi=\alpha_k\sigma\alpha_{k-1}\dots \alpha_2\sigma\alpha_1,$$
where $\sigma$ is the standard quadratic involution and $\alpha_1,\dots,\alpha_k\in \Aut(\p^2)=\PGL(3,\k)$ (See for example \cite[Chapter V, \S5, Theorem 2, p. 100]{Sha} or \cite[Chapter 8]{alberich}).

For each $i$, we write $\varphi_i=\alpha_i \sigma \alpha_{i-1}\dots \alpha_2\sigma\alpha_1$, and observe that $\varphi_1$ is linear and $\varphi_k=\varphi$. The homaloidal type of $\varphi_i$ is obtained from the one of $\varphi_{i-1}$ by applying $\sigma_0$ and a permutation of coordinates. Hence, the homaloidal type of $\varphi$ belongs to $W(e_0)$.

We now take an element $f\in W$ that we can write as  
$$f=a_k\sigma_0a_{k-1}\dots a_2\sigma_0a_1,$$
where $\sigma_0$ is the automorphism of $V$ defined before and $a_1,\dots,a_k$ are permutations of the $e_i$ fixing $e_0$. We prove by induction on $k$ that the type of $f(e_0)$ satifies assertion $(ii)$, which will achieve the proof.

If $k=1$ or $k=2$, the result is obvious, as the type is of degree $1$ or of degree $2$ with three simple base-points. 
We can then assume $k>2$ and show the result using induction hypothesis.

In order to simplify the proof, we will assume that $a_k$ is the identity, since permutations of points does not change the result of $(ii)$. We then write $f(e_0)=(d,m_1,\dots,m_r)$ and $f'=a_{k-1}\sigma_0\dots a_2\sigma_0a_1$, which implies that
$$f'(e_0)=(2d-m_1-m_2-m_3,d-m_2-m_3,d-m_1-m_3,d-m_1-m_2,m_4,\dots,m_r).$$
Using induction hypothesis, we obtain a dense open subset $U\subset (\p^2)^r$ such that for each $(p_1,\dots,p_r)\in U$, there exists a birational transformation $\varphi$ having degree $2d-m_1-m_2-m_3$ and which base-points are the $p_i$ with multiplicities given by $f'(e_0)$. This follows from the induction hypothesis and the fact that we can add general points if one of the multiplicities $d-m_2-m_3,d-m_1-m_3,d-m_1-m_2$ is zero.

We denote by $V\subset U$ the open subset where no three of the points $p_i$ are collinear. For each $(p_1,\dots,p_r)\in V$ we take the map $\varphi$ associated to these points and define $\hat{\varphi}=\varphi\psi$ where $\psi=\tau\sigma\tau^{-1}$ and  $\tau\in \Aut(\p^2)$ sends the three base-points of $\sigma$ onto $p_1,p_2,p_3$. Then $\hat{\varphi}$ is a birational map of degree $d$ with multiplicities $m_1,\dots,m_r$ at $p_1,p_2,p_3,\psi(p_4),\dots,\psi(p_r)$ respectively. We can the define $\hat{V}\subset (\p^2)^r$ as the open set $$\hat{V}=\{(q_1,\dots,q_r)\mid (q_1,q_2,q_3,\psi(q_4),\dots,\psi(q_r))\in V\};$$
that concludes the proof.
\end{proof}
\begin{remark}\label{Rem:MatrixInv}
Following Proposition~\ref{Prop:HudsonW}, we can associate to each element $\varphi\in \Bir(\p^2)$ an element $g\in W$, unique up to permutations at source and target, such that $g(e_0)$ corresponds to the type of $\varphi$. The matrix $g$ corresponds to the characteristic matrix studied in \cite[Chapter 5]{alberich} and gives the curves contracted by $g$ and its inverse. In particular, $g^{-1}\in W$ is the map associated to $\varphi^{-1}$, so the homaloidal type of $\varphi^{-1}$ is obtained by computing $g^{-1}(e_0)$.
\end{remark}

Using Proposition~\ref{Prop:HudsonW}, one obtain the classical algorithm (Hudson's test) that decides whether a homaloidal type is proper or improper. Let us recall how it works: 
\begin{enumerate}\item
Taking a homaloidal type $(d;m_1,\dots,m_r)$ with $d\ge 2$, and all integers $m_i$ non-negative and order them so that  $m_1\ge m_2\ge m_3\ge \dots \ge m_r$.
\item
We then replace $(d;m_1,\dots,m_r)$  with $(d-\epsilon,m_1-\epsilon,m_2-\epsilon,m_3-\epsilon,m_4,\dots,m_r)$, where $\epsilon=m_1+m_2+m_3-d$, and then go back an apply the first step.
\item We end when we reach $(1;0,\dots,0)$, in which case the test is fullfilled, or when at least one $m_i$ is negative, in which case the test is not fullfilled.
\end{enumerate}
Then, we recall the following result.
\begin{lemma}\label{lem:proper} A homaloidal type is proper if and only if it satisfies Hudson's test.
\end{lemma}
\begin{proof}
We observe first that if $(d;m_1,\dots,m_r)$ is a homaloidal type with $d\ge 2$, $m_1\ge m_2\ge \dots \ge m_r\ge 0$, then $m_1+m_2+m_3\ge d+1$. This was already observed by Noether and is a direct consequence of the Noether equalities (see for example \cite[Corollary 2.6.7, page~55]{alberich}). Hence, the integer $\epsilon$ in the test above is always non-negative. 

If $d=1$, the only possibilities for the $m_i$ is to be zero. Hence, the algorithm above always has a end: either we reach $d=1$ with all $m_i$ being zero or at some point one $m_i$ is negative.

Note that Hudson's test consists of applying elements of $W$ to $de_0-\sum m_ie_i$. If the test is fulfilled, then the type is in the image of $W$, and is thus a proper homaloidal type by Proposition~\ref{Prop:HudsonW}. If the test is not fullfilled, then we finish with a type which is improper as it contains a negative integer. The type from which we started is then in the same orbit as this one by $W$ and is thus improper by Proposition~\ref{Prop:HudsonW}.
\end{proof}

\begin{remark}\label{rem:Hudson}
One may stop Hudson's test as soon as one reaches in step (1) a homaloidal type which is already known to be either proper or improper, and accordingly the test is either fulfilled or not fulfilled.
\end{remark}
\begin{remark}\label{Rem:MatrixInv2}
Applying Hudson's test to the type of a birational map $\varphi$ we also obtain the matrix of $\varphi$, which is the element of $W$ corresponding to it (Remark~\ref{Rem:MatrixInv}). This shows in particular that the homaloidal type of $\varphi^{-1}$ only depends on the homaloidal type of $\varphi$ and not of the position of the base-points, and also provides a method to compute the type of the inverse (already explained in \cite[Definition 5.4.24, page 156]{alberich}).\end{remark}

\begin{example}\label{Exa:Types}
Applying the algorithm corresponding to Hudson's test, one easily finds all linear systems of small degree. We use the notation $m^r$ to write $m,\dots,m$ ($r$ times).

For each degree $d\ge 2$, the type $(d;d-1,1^{2d-2})$ is a de Jonqui\`eres homaloidal type. For $d\ge 4$, we also have another type, which is $(d;d-2,2^{d-2},1^3)$. In degree $d\le 11$, all other proper homaloidal types are given in Table~\ref{tab:Hudson11}.
\begin{table}[h]
$$
\begin{array}{llll}
\hline\hline
\vpb
(5;2^6)\\
\hline
\vpb
(6;3^3,2,1^4)& (6;3^2,2^4,1)\\
\hline
\vpb
(7;4,3^3,1^5)&(7;4,3^2,2^3,1^2)& (7;3^4,2^3)\\
\hline
\vpb
(8;5,3^3,2^2,1^3)&(8;5,3^2,2^5)&(8;4^3,3,1^6)&(8;4^3,2^3,1^3)\\
\vpb
(8;4^2,3^2,2^3,1)&(8;4,3^5,1^2)&(8;3^7)\\
\hline
\vpb
(9;6,3^4,2,1^4)& (9;6,3^3,2^4,1)& (9;5,4^3,1^7)& (9;5,4^2,3,2^3,1^2)\\
\vpb
(9;5,4,3^4,1^3)& (9;5,4,3^3,2^3)&(9;4^4,2^4)&(9;4^3,3^3,2,1)\\
\hline
\vpb
(10;7,3^5,1^5)& (10;7,3^4,2^3,1^2) & (10;6,4^3,2^3,1^3)& (10;6,4^2,3^3,1^4)\\
\vpb
(10;6,4^2,3^2,2^3,1) & (10;6,3^7) &(10;5^3,4,1^8) & (10;5^3,3,2^3,1^3)\\
\vpb
(10;5^3,2^6)&(10;5^2,4^2,2^4,1)&  (10;5^2,4,3^3,2,1^2)&(10;5^2,3^5,2)\\
\vpb
(10;5,4^3,3^2,2^2)&(10;4^6,1^3)&(10;4^5,3^2,1)&\\
\hline
\vpb
(11;8,3^5,2^2,1^3)&(11;8,3^4,2^5)&(11;7,4^3,3^2,1^5)&(11;7,4^3,3,2^3,1^2)\\
\vpb
(11;7,4^2,3^3,2^3)&(11;7,4,3^6,1) & (11;6,5^3,1^9)&(11;6,5^2,4,2^4,1^2)\\
\vpb 
(11;6,5^2,3^3,2,1^3)&(11;6,5^2,3^2,2^4)&(11;6,5,4^2,3^2,2^2,1)&
(11;6,4^5,1^4)\\
\vpb
(11;6,4^3,3^4)& (11;5^4,2^5)& (11;5^3,4,3^3,1^2) &(11;5^2,4^4,2,1^2)\\
\vpb
(11;5^2,4^3,3^2,2)\\
\hline\hline
\end{array}$$
\caption{Proper homaloidal types of degree $d\le 11$ which are not of type 
$(d;d-1,1^{2d-2})$ or $(d;d-2,2^{d-2},1^3)$.
}
\label{tab:Hudson11}
\end{table}
It can also be checked that these types are the same as in \cite[Table I, pages 437--438]{Hudson}.

\end{example}

In Section \ref{Soneless} we will need other proper homaloidal types in each degree.

\begin{example} \label{ex3m}
For each integer $m \geq 3$, the following homaloidal types
\begin{align} \label{3m}
 & (3m;3m-6,6^{m-3},4^3,3^2,2,1),
\\ \label{3m+1}
& (3m+1;3m-5,6^{m-2},4,3^3,1^4),
\\ \label{3m+2}
& (3m+2;3m-4,6^{m-2},4^2,3^2,2^2,1 ),
\end{align}
are proper \cite[Table II]{Hudson}. This can also be shown by applying the Hudson's test for $m=3$ and $m=4$ and then apply induction on $m$: running once step (2) of the algorithm, the properness of the homaloidal types \eqref{3m}, \eqref{3m+1}, \eqref{3m+2} for $m$ proves the properness of these types for $m+2$.
\end{example}

\section{Varieties parametrising birational maps of small degree}
\subsection{The topology on $\Bir(\p^2)$}\label{SubSec:TopologyCremona}
We recall the notion of families of birational maps, introduced by M.~Demazure in \cite{De} (see also \cite{Se}, \cite{Bl}).

\begin{definition} \label{Defi:Family}
Let $A,X$ be irreducible algebraic varieties, and let $f$ be a $A$-birational map of the $A$-variety $A\times X$, inducing an isomorphism $U\to V$, where $U,V$ are open subsets of $A\times X$, whose projections on $A$ are surjective.

The rational map $f$ is given by $(a,x)\dasharrow (a,p_2(f(a,x)))$, where $p_2$ is the second projection, and for each $\k$-point $a\in A$, the birational map $x\dasharrow p_2(f(a,x))$ corresponds to an element  $f_a\in \Bir(X)$.
The map $a\mapsto f_a$ represents a map from $A$ $($more precisely from the $A(k)$-points of $A)$ to $\Bir(X)$, and will be called a \emph{morphism} from $A$ to $\Bir(X)$.
\end{definition}
These notions yield the natural Zariski topology on $\Bir(X)$, introduced implicitly by M.~Demazure \cite{De} and explicitly by J.-P. Serre \cite{Se}:
\begin{definition}  \label{defi: Zariski topology}
A subset $F\subseteq \Bir(X)$ is closed in the Zariski topology
if for any algebraic variety $A$ and any morphism $A\to \Bir(X)$ the preimage of $F$ is closed.
\end{definition}
\begin{remark}\label{Rem:TopInv}
Any birational map $X\dasharrow Y$ yields a homeomorphism between $\Bir(X)$ and $\Bir(Y)$, and for any $\varphi\in \Bir(X)$ the maps $\Bir(X)\to \Bir(X)$ given by $\psi\mapsto \psi \circ \varphi$, $\psi\mapsto \varphi\circ \psi$ and $\psi\mapsto \psi^{-1}$ are homeomorphisms. \end{remark}

\begin{remark}
In the sequel, the topology on $\Bir(\p^2)$ and its subsets will always be the Zariski topology given in Definition~\ref{defi: Zariski topology}.
\end{remark}

\subsection{The varieties $W_d$,  $\Bir_d$ and $\Bir_d^{\circ}$}
Let us recall the following notation, which is taken from \cite[Definition 2.3]{BF} and \cite[p.~1112]{BCM}.
\begin{definition}\label{DefWHG}
Let $d$ be a positive integer.
\begin{enumerate}
\item
We define $W_d$ to be the set of equivalence classes of non-zero  triples $(h_0,h_1,h_2)$
of homogeneous polynomials $h_i\in \k[x,y,z]$ of degree $d$,
where $(h_0,h_1,h_2)$ is equivalent to $(\lambda h_0,\lambda h_1,\lambda h_2)$ for any $\lambda\in \k^{*}$.
The equivalence class of $(h_0,h_1,h_2)$ will be denoted by $[h_0:h_1:h_2]$.
\item
We define $\mathrm{Bir}_d\subseteq W_d$ to be the set of elements $h=[h_0:h_1:h_2]\in W_d$
such that the rational map
$\psi_h\colon \p^2\dasharrow \p^2$ given by $$[x:y:z]\dashmapsto
[h_0(x,y,z):h_1(x,y,z):h_2(x,y,z)]$$ is birational.
We denote by $\pi_d$ the map $\mathrm{Bir}_d\to \Bir(\p^2_\k)$ which sends $h$ onto $\psi_h$.
\item
We define by $\mathrm{Bir}_d^{\circ}\subseteq \mathrm{Bir}_d$ the subset of elements $[h_0:h_1:h_2]\in \mathrm{Bir}_d$ such that the polynomials $h_0,h_1,h_2$ have no common factor of degree $\ge 1$.
\end{enumerate}
\end{definition}
\begin{remark}
Note that $\mathrm{Bir}_d$ is the notation of \cite{BCM} and was called $H_d$ in \cite{BF}.
\end{remark}
\begin{remark}
The map $\pi_d$ is not injective for $d\ge 2$ but restricts to a natural bijection between $\mathrm{Bir}_d^{\circ}$ and the set $\Bir(\p^2)_{d}$ of maps of degree $d$.
\end{remark}

\begin{lemma} \label{lem:WHalgebraic}
Let $W_d,\Bir_d$ be as in Definition~$\ref{DefWHG}$.
Then, the following holds:

\begin{enumerate}[$(1)$]
\item
The set $W_d$ is  isomorphic to $\p^{r}$,
where $r=3 \binom{d+2}{2}-1 =  3d(d+3)/2+2$.
\item
The set $\Bir_d$ is locally closed in $W_d$, and thus inherits from $W_d$ the structure of an algebraic variety. 
\item
The map $\pi_d\colon \Bir_d\to \Bir(\p^2)$ is a morphism, which is continuous and closed. Its image is the set $\Bir(\p^2)_{\le d}$ of birational transformations of degree $\le d$.
\item
For any $\varphi \in \Bir(\p^n)_{\le d}$, the set $(\pi_d)^{-1}(\varphi)$ is closed in $W_d$ $($hence in $\Bir_d)$.
\item
The set $\Bir_d^{\circ}$ is open in $\Bir_d$.\end{enumerate}
\end{lemma}
\begin{proof}
Follows from \cite[Lemma 2.4, Corollary 2.9 and Proposition 2.15]{BF}.
\end{proof}
\begin{corollary}\label{Coro:ClosureBird0}
Let $d\ge 1$. We have an equality
$$\pi_d(\overline{\Bir_d^{\circ}})=\overline{\Bir(\p^2)_{d}},$$
where the closure of $\Bir_d^{\circ}$ is taken in $\Bir_d$ and the closure of $\Bir(\p^2)_d$ is taken in $\Bir(\p^2)$.
\end{corollary}
\begin{proof}
Follows from the fact that $\pi_d(\Bir_d^{\circ})=\Bir(\p^2)_{d}$ and that $\pi_d\colon \Bir_d\to \Bir(\p^2)$ is closed and continuous.
\end{proof}
\begin{lemma}\label{Lemm:InverseHomeo}
The map $$\begin{array}{rcl}
\Bir(\p^2)&\to& \Bir(\p^2)\\
\varphi&\mapsto & \varphi^{-1}\end{array}$$
is a homeomorphism, which sends $\Bir(\p^2)_d$ onto itself for each $d$.

In particular, an element $\varphi\in \Bir(\p^2)$ belongs to the closure of $\Bir(\p^2)_d$ if and only if $\varphi^{-1}$ belongs to the closure of $\Bir(\p^2)_d$.
\end{lemma}
\begin{proof}
Follows from the definition of the topology of $\Bir(\p^2)$ (see Remark~\ref{Rem:TopInv}), and from the fact that the degree of an element and its inverse are the same.
\end{proof}

\begin{proposition}\label{prop:BCM}
Let $d\ge 2$. 

\begin{enumerate}[$(1)$]
\item
If $I\subset \Bir_d^{\circ}$ is an irreducible component, there exists a proper homaloidal type $\Lambda=(d;m_1,\dots,m_r)$ such that a general element of $I$ is a  birational map $\varphi$ of $\p^2$ of type $\Lambda$, such that neither $\varphi$ nor $\varphi^{-1}$ have infinitely near base-points. Moreover, all birational maps of type $\Lambda$ are contained in $I$.
\item
The association in $(1)$ yields a one-to-one correspondence between the irreducible components of $\Bir_d^{\circ}$ and the proper homaloidal types of degree $d$.\end{enumerate}
\end{proposition}

\begin{proof}
It follows from \cite[Theorem 1 and Lemma 36]{BCM}.
\end{proof}
\begin{notation}\label{Notation:BirLambda}
If $\Lambda=(d;m_1,\dots,m_r)$ is a proper homaloidal type, we will follow \cite[Definition 37]{BCM} and denote by $\Bir_\Lambda^{\circ}\subset \Bir_d^{\circ}$ the irreducible component of $\Bir_d^{\circ}$ whose general element is of type $\Lambda$.
\end{notation}
Even if all birational maps of type $\Lambda$ belong to $\Bir_\Lambda^{\circ}$, not all elements of $\Bir_\Lambda^{\circ}$ have homaloidal type $\Lambda$. Indeed, some map can belong to two or more different irreducible components. In particular, $\Bir^{\circ}_d$ is connected if $d\le 6$ \cite[Theorem 2]{BCM}.

\begin{notation}\label{Notation:LambdaDual}
If $\Lambda=(d;m_1,\dots,m_r)$ is a proper homaloidal type, we will denote by $\Lambda^*$ the homaloidal type such that the inverse of a map of type $\Lambda$ has type $\Lambda^*$ (as observed in Remark~\ref{Rem:MatrixInv2}, the homaloidal type of the inverse of $\varphi^{-1}$ only depends on the homaloidal type of $\varphi\in \Bir(\p^2)$).
\end{notation}
\begin{remark}\label{Rem:DualHomType}
The map $\varphi\mapsto \varphi^{-1}$ yields a homeomorphism of $\Bir_d^{\circ}$ (see Lemma~\ref{Lemm:InverseHomeo}). In particular, it sends an irreducible component $\Bir^{\circ}_\Lambda$ onto $\Bir^{\circ}_{\Lambda^*}$.
\end{remark}
\begin{remark}\label{Rem:Degreesmallinverse}
In degree $d\le 5$, every proper homaloidal type $\Lambda$ satisfies $\Lambda=\Lambda^*$, but this is not true in degree $6$, where 
$(6;4,2^4,1^3)^*=(6;3^3,2,1^4)$ (see \cite[Table I, pages 437--443]{Hudson} for the description in each degree $d\le 16$). 

One can moreover observe that for each $d\ge 6$ there are types which are self-dual and types which are not: 
\begin{enumerate}
\item
 $(d;d-1,1^{2d-2})^*=(d;d-1,1^{2d-2})$ for $d\ge 2$;
 \item
 $(d;d-2,2^{d-2},1^3)^{*}=(d;\lfloor\frac{d}{2}\rfloor^3,\lceil \frac{d-1}{2}\rceil,1^{d-2})\not=(d;d-2,2^{d-2},1^3)$ for $d\ge 6$.
 \end{enumerate}
\end{remark}

\subsection{Jacobians and curves contracted}
The curves contracted by birational maps and the Jacobian of the maps will be useful to parametrise subvarieties of $\Bir_d$ and $\Bir_d^{\circ}$, and to obtain results on possible degenerations.
\begin{definition}
If $f=[f_0:f_1:f_2]\in \Bir_d$, we denote by $J(f)$ the polynomial, defined up to multiple by a constant of $\k^*$, which is the determinant of the matrix of partial derivatives of $f_0,f_1,f_2$ with respect to $x,y,z$. It is the \emph{Jacobian} of $f$. This gives a  morphism $J\colon\Bir_d\to \p(\k[x,y,z]_{3(d-1)})$.
\end{definition}

We now introduce a new definition, that we will then use to study degenerations of maps (together with Definition~\ref{Def:BirLambda}).
\begin{definition}
Let $f=[f_0:f_1:f_2]\in \Bir_d$, let $h\in \k[x,y,z]$ be a homogeneous polynomial and let $q=[q_0:q_1:q_2]\in \p^2$. 

$(i)$ We say that $f$ \emph{contracts $h$ onto $q\in \p^2$} if $q_if_j-q_jf_i$ is a multiple of $h$ for each $i,j\in \{0,1,2\}$.

$(ii)$ We say that $q$ \emph{is a base-point of $f$ of multiplicity $k$} if a general linear combination of $f_0,f_1,f_2$ has multiplicity $k$ at $q$.
\end{definition}
\begin{remark}
If $h$ is a factor of $\gcd(f_0,f_1,f_2)$, then $h$ is contracted onto any point of~$\p^2$. But otherwise, there is only one possible point where $h$ can be contracted.

If $f\in \Bir_d$ and $\varphi=\pi_d(f)\in \Bir(\p^2)_{\le d}$ is the corresponding birational map, every base-point of $\varphi$ is a base-point of $f$. But if $\varphi$ has degree $<d$, then $f$ has infinitely many base-points, which correspond to the points of the common factor of $f_0,f_1,f_2$.
\end{remark}

\bigskip

Let us recall the following classical result.
\begin{lemma}\label{Lemm:DecompJac}
Let $f=[f_0:f_1:f_2]\in \Bir_d^{\circ}$  be an element which corresponds to the birational map $\pi_d(f)=\varphi\in \Bir(\p^2)_d$ and denote by $(d;m_1,\dots,m_r)$ the homaloidal type of $\varphi^{-1}$. We also denote by $g=[g_0:g_1:g_2]\in \Bir_d^{\circ}$ the element corresponding to $\varphi^{-1}$, and assume that $\varphi^{-1}$ has no infinitely near base-point.
\begin{enumerate}[$(1)$]
\item
If  $h\in \k[x,y,z]$ is a homogeneous polynomial which is contracted by $f$ onto a point $q\in \p^2$, then each point of the curve of $\p^2$ given by $h=0$ which is not a base-point of $\varphi$ is sent by $\varphi$ onto $q$.
Moreover, $h$ is a divisor of the Jacobian $J(f)$.
\item
The Jacobian $J(f)$ admits a decomposition into $J(f)=\prod_{i=1}^r p_i$, where  $p_1,\dots,p_r$ are homogeneous polynomials of degree $m_1,\dots,m_r$ respectively, each of them contracted by $f$ onto points $q_1,\dots,q_r\in \p^2$ respectively, all being base-points of $\varphi^{-1}$ of multiplicity equal to $m_1,\dots,m_r$ respectively.

Moreover, the following hold:
\begin{enumerate}[$(a)$]
\item
The points $q_i$ are pairwise distinct.
\item
The points $q_1,\dots,q_r$ are the base-points of $\varphi^{-1}$.
\item
Each $p_i$ is an irreducible polynomial.
\item
The decomposition $J(f)=\prod_{i=1}^r p_i$ corresponds to the decomposition of $J(f)$ into irreducible polynomials.
\end{enumerate}
\end{enumerate}
\end{lemma}
\begin{proof}
Follows from \cite[Proposition 3.5.3 and Theorem 3.5.6]{alberich}. 
\end{proof}

\begin{definition}\label{Def:BirLambda}
Let $\Lambda$ be a homaloidal type, such that $\Lambda^{*}=(d;m_1,\dots,m_r)$. We denote by $$\Bir_{\Lambda}\subset \Bir_d$$ the set of elements $f=[f_0:f_1:f_2]\in \Bir_d$ such that there exist $g=[g_0:g_1:g_2]\in \Bir_d$ and homogeneous polynomials $p_1,\dots,p_r$ of degree $m_1,\dots,m_r$ respectively, each of them contracted by $f$ onto points $q_1,\dots,q_r$, being base-points of $g$ of multiplicity at least $m_1,\dots,m_r$, and such that  
$J(f)=\prod\limits_{i=1}^r p_i$ and $\pi_d(g)\circ \pi_d(f)$ is the identity.
\end{definition}
The following result shows that this definition is consistent with Notation~\ref{Notation:BirLambda}.
\begin{proposition}\label{Prop:BirLclosed}Let $\Lambda$ be a proper homaloidal type of degree $d\ge 2$. Then, the following hold:
\begin{enumerate}[$(1)$]
\item The set $\Bir_{\Lambda}$ is closed in $\Bir_d$.
\item $\Bir_{\Lambda}^{\circ}$ is the unique irreducible component of $\Bir_d^{\circ}$ contained in $\Bir_{\Lambda}\cap \Bir_d^{\circ}$.\end{enumerate}
\end{proposition}
\begin{proof}
$(1)$ We write $\Lambda^{*}=(d;m_1,\dots,m_r)$ and prove that $\Bir_{\Lambda}$ is closed in $\Bir_d$. To do this, we denote by $X_\Lambda$ the subset of $$\Bir_d \times W_d\times \p(\k[x,y,z]_{m_1})\times \dots\times \p(\k[x,y,z]_{m_r})\times (\p^2)^r$$ consisting of elements $$(f,g,p_1,\dots,p_r,q_1,\dots,q_r)$$ such that 
\begin{itemize}
\item
The equality $J(f)=\prod\limits_{i=1}^r p_i$ holds in $\p(\k[x,y,z]_{3d-3})$.
 \item
The element $[g_0(f_0,f_1,f_2): g_1(f_0,f_1,f_2): g_2(f_0,f_1,f_2)]\in W_{d^2}$ is equal to $[xh:yh:zh]$ for some element $h\in \k[x,y,z]_{d^2-1}$; this corresponds to ask that $g_i(f_0,f_1,f_2)x_j=g_j(f_0,f_1,f_2)x_i$ for each $i,j$, where $x_0=x$, $x_1=y$, $x_2=z$.
 \item For each $j$, the polynomial $p_j$ is contracted by $f$ onto $q_j$.
 \item For each $j$, the point $q_j$ is a base-point of $h$ of multiplicity $\ge m_j$.
 \end{itemize} 
 Since all above conditions are closed, the set $X_\Lambda$ is closed in $\Bir_d \times W_d\times \p(\k[x,y,z]_{m_1})\times \dots\times \p(\k[x,y,z]_{m_r})\times (\p^2)^r$. Moreover, the second condition is equivalent to ask that $g\in \Bir_d$ and that $\pi_d(f)\circ \pi_d(g)$ is the identity. Hence, the set  $\Bir_\Lambda$ is the projection of $X_\Lambda$ onto $\Bir_d$ and is thus closed in $\Bir_d$, because $W_d\times \p(\k[x,y,z]_{m_1})\times \dots\times \p(\k[x,y,z]_{m_r})\times (\p^2)^r$ is projective.

 If $f=[f_0:f_1:f_2]\in \Bir_d^{\circ}$ is of homaloidal type $\Lambda$ such that $(\pi_d(f))^{-1}$ has no infinitely near base-points, then $f$ belongs to $\Bir_{\Lambda}$ and not to any other $\Bir_{\Lambda'}$ (Lemma~\ref{Lemm:DecompJac}). This shows, together with Proposition~\ref{prop:BCM}, that $\Bir_{\Lambda}^{\circ}\subset \Bir_{\Lambda}\cap \Bir_d^{\circ}$ and that $\Bir_\Lambda\cap \Bir_d^{\circ}$ does not contain any other irreducible component. 
\end{proof}
\begin{remark}
%
Each element of $\overline{\Bir_d^{\circ}}\subset \Bir_d$ is contained in a $\Bir_\Lambda$, for some homaloidal type $\Lambda$ of degree $d$. This will give some conditions on the elements of $\overline{\Bir_d^{\circ}}$, and thus on the set $\overline{\Bir(\p^2)_d}=\pi_d(\overline{\Bir_d^{\circ}})$ (Corollary~\ref{Coro:ClosureBird0}). 

The set $\overline{\Bir_\Lambda^{\circ}}$ is contained in  $\Bir_\Lambda$, but we do not know if equality holds. We also do not know if $\Bir_\Lambda^{\circ}=\Bir_\Lambda\cap \Bir_d^\circ$.
\end{remark}

\begin{example}
Let us consider $$\begin{array}{rcll}
f&=&[(x+y-z)y(3y-z): (2y-z)x(3y-z): (2y-z)xy]&\in \Bir_3^{\circ}\\
g&=&[(y-2z)yz: (xy-xz-yz)z: (xy-xz-yz)(y-3z)]&\in \Bir_3^{\circ}\end{array}$$
which are such that $\pi_3(f)\circ \pi_3(g)$ is the identity and $J(f)=3xy(3y-z)(z-y)(2y-z)^2.$ The polynomials $x$, $y$, $3y-z$, $z-y$ and $2y-z$ are contracted respectively onto $[1:0:0]$, $[0:1:0]$, $[0:0:1]$, $[2:2:1]$, $[1:0:0]$.

There are multiple ways to choose the polynomials $p_1,\dots,p_5$ and the points $q_1,\dots,q_5$ and in each way there is a polynomial $p_i$ of degree $1$ contracted onto $q_i=[1:0:0]$, which is a base-point of $g$ of multiplicity $2$. The fact that the multiplicity is higher than the degree of the polynomial is because this polynomial corresponds in fact to a base-point of $g$ infinitely near to $[1:0:0]$, having multiplicity $1$.
\end{example}

\section{Existence of degenerations}
\subsection{Degeneration associated to two base-points}
We first prove the following simple degeneration lemma.

\begin{lemma}\label{Lemm:QuadraticDeg}
Let $p_1\in \p^2$ and let $p_2$ be a point which is either in the first neighbourhood of $p_1$ or a distinct point of $\p^2$.

Then, there exists a morphism $\nu\colon\A^1\to \Bir(\p^2)$ and a morphism $p_3\colon \A^1\to \p^2$ such that the following hold:

\begin{enumerate}[$(1)$]
\item
For $t\not=0$, $\nu(t)$ is a quadratic map with base-points $p_1,p_2,p_3(t)$;
\item
The map $\nu(0)$ is the identity and $p_3(0)$ is collinear with $p_1$ and $p_2$.
\end{enumerate}
\end{lemma}
\begin{remark}
In this degeneration, the linear system of conics through $p_1,p_2,p_3(t)$ degenerates to a system of conics through three collinear points, which is the union of the line through the points and the system of lines of $\p^2$.
\end{remark}
\begin{proof}
We first assume that $p_1,p_2$ are proper points of $\p^2$ and can then assume that $p_1=[1:0:0]$ and $p_2=[0:1:0]$. Then, we consider the morphism $\kappa\colon \A^1\to \Bir(\p^2)$ given by
$$\kappa(t)\colon [x:y:z]\dashmapsto [(ty-z)x: (tx-z)y: (tx-z)(ty-z)].$$
For $t\not=0$, $\kappa(t)$ is a quadratic birational involution with base-points $p_1,p_2, [1:1:t]$. Moreover, $\kappa(0)$ equal to the linear automorphism $[x:y:z]\mapsto [x:y:-z]$. We can then define $\nu$ as $\nu(t)=\kappa(t)\circ\kappa(0)$.

The second case is when $p_2$ is infinitely near to $p_1$. We can then assume that $p_1=[1:0:0]$ and that $p_2$ corresponds to the tangent direction $z=0$. In this case, we choose $\kappa\colon \A^1\to \Bir(\p^2)$ given by
$$\kappa(t)\colon [x:y:z]\dashmapsto [-xz+ty^2:yz: z^2].$$
For $t\not=0$, $\kappa(t)$ is a quadratic birational involution with base-points $p_1,p_2$ and some point $p_3(t)$ infinitely near $p_2$. Moreover, $\kappa(0)$ equal to the linear automorphism $[x:y:z]\mapsto [-x:y:z]$. Again, choosing  $\nu(t)=\kappa(t)\circ\kappa(0)$ works.
\end{proof}
\begin{proposition} \label{prop:mi+mj=d-k}
Suppose that $\gamma\in \Bir(\p^2)$ is a birational map of degree $d$ and has two base-points $p_1,p_2$ of multiplicity $m_1,m_2$ with $m_1+m_2=d-k$, such that $p_1$ is a proper point of $\p^2$ and $p_2$ is either a proper point of $\p^2$ or in the first neighbourhood of $p_1$.

Then, there exists a morphism $\rho\colon\A^1\to \Bir(\p^2)$ such that $\rho(0)=\gamma$ and $\rho(t)$ has degree $d+k$ for a general $t\not=0$.
\end{proposition}
\begin{proof}
We use the morphism $\nu\colon \A^1\to \Bir(\p^2)$ given by Lemma~\ref{Lemm:QuadraticDeg} and define $\rho$ as $\rho(t)=\gamma\circ \nu(t)^{-1}$. By construction, this is a morphism which satisfies $\rho(0)=\gamma$. Moreover, the degree of the map $\rho(t)$ for a general $t$ is equal to $2d-m_1-m_2=d+k$.
\end{proof}
\begin{remark}
If $p_1,\dots,p_r$ are the base-points of $\gamma$ of multiplicity $m_1,\dots,m_r$, the degeneration provided by Proposition~\ref{prop:mi+mj=d-k} gives a family of birational maps of degree $d+k$ with base-points of multiplicity $m_1+k,m_2+k,k,m_3,\dots,m_r$. The point of multiplicity $k$ created degenerates to a point collinear with the first two points, and the linear system becomes the union of the linear system of $\gamma$ with $k$ times the line through $p_1$ and $p_2$.
\end{remark}
In order to be able to apply Proposition~\ref{prop:mi+mj=d-k}, we need to compute the multiplicities of birational maps and estimate the integer $k$ which appears in the statement. This is done in the following lemma.
\begin{lemma}\label{Lem:MultiplicitiesDminus1}
Let $\varphi$ be a birational map of degree $d$. Then there exists two distinct points $($proper or infinitely near$)$ of multiplicity $m_1,m_2\ge 0$, such that $$\begin{array}{cll}
m_1+m_2=d-1&\mbox{ if }d\in \{ 1,2,\ldots,6,7,9,11 \}\\
d-2\le m_1+m_2\le d-1&\mbox{ if }d=8\\
d-3\le m_1+m_2\le d-1&\mbox{ if }d=10\\
\frac{2d}{3}<m_1+m_2<d& \mbox{ if }d\ge 12\end{array}$$
\end{lemma}
\begin{proof}
If $\varphi$ is of de Jonqui\`eres type, we can choose $m_1=d-1$ and $m_2=0$. If $\varphi$ has an homaloidal type $(d;d-2,2^{d-2},1^3)$, we can choose $m_1=d-2$ and $m_2=1$.

If $d\in \{ 1,2,\ldots,6,7,9,11 \}$, we find two base-points of multiplicity $m_1,m_2$ with $m_1+m_2=d-1$ (see Example~\ref{Exa:Types}). For $d=8,10$, the result also follows from  Example~\ref{Exa:Types}.
 
 We can thus assume $d\ge 12$, and find, with Noether inequalities, two points of multiplicity $m_1,m_2$ with $m_1+m_2>\frac{2d}{3}$.
If $m_1+m_2<d$, we are done, so we can assume that $m_1+m_2=d$ ($m_1+m_2>d$ is not possible by the B\'ezout theorem), i.e.\ $m_2=d-m_1$.
Either $m_1>d/2$, or $m_1=m_2=d/2$.
Moreover, Noether inequalities implies also that $m_3>(d-m_1)/2=m_2/2$ \ \cite[Lemma 8.2.6]{alberich}.
Hence, $m_1+m_3>(d+m_1)/2\geq 3d/4>2d/3$, that is the assertion, unless $m_3=m_2=d-m_1$ too.

Let $\gamma$ be the number of points, different from $p_1$, with multiplicity $d-m_1$.
Either $m_1>d/2$, or $m_1=m_2=\cdots=m_{\gamma+1}=d/2$.

In the latter case, applying a quadratic map centered at $p_1,p_2,p_3$, one finds a Cremona map of degree $d/2$ that must have a base-point $q_4$ of multiplicity $m_4$ with $d/2>m_4> (d/2)/3=d/6$. This means that also $\varphi$ has a point $p_4$ of multiplicity $m_4$. It follows that $d>m_1+m_4> 2d/3$ and the proof is concluded in this case. 

In the former case, we claim that, if $\varphi$ is not de Jonqui\`eres (in which case the assertion of the lemma is trivial, as we already observed), then $\varphi$ has at least one further base-point $p_{\gamma+2}$ (cf.\ \cite[p.\ 75]{Hudson}).
Suppose indeed that the number $r$ of base-points is $\gamma+1$. 
Multiplying the first Noether equality in \eqref{Neq} by $m_2$ and subtracting the second Noether equality in \eqref{Neq}, we find
\[
\sum_{i=1}^r m_i(m_2-m_i)=3m_2(d-1)-(d^2-1)=(d-1)(2d-3m_1-1)
\]
that is
\[
m_1(d-2m_1)=(d-1)(2d-3m_1-1)
\]
which is impossible because $m_1<d-1$.
So our claim is proved and there is at least another base-point $p_{\gamma+2}$ that we can use together with $p_1$.
Note also that the assertion is trivial if $m_1+1 > 2d/3$, i.e.\ $m_1>(2d-3)/3$.
Hence, we may assume that $m_1\leq(2d-3)/3$ and therefore $m_2=d-m_1\geq (d+3)/3$.
Recalling that $m_2=m_3=d-m_1<d/2$, it follows that $d>m_2+m_3>2d/3$, as wanted.
\end{proof}


\begin{corollary}\label{Coro:ddp}
We have $\Bir(\p^2)_d\subset \overline{\Bir(\p^2)_{d+1}}$ for each $$d \in \{ 1,2,\ldots,6,7,9,11 \}$$
and $\Bir(\p^2)_{8}\subset \overline{\Bir(\p^2)_{10}}$.
\end{corollary}
\begin{proof}Let $\gamma\in \Bir(\p^2)$ be of degree $d\in \{ 1,2,\ldots,6,7,9,11 \}$. 
By Lemma~\ref{Lem:MultiplicitiesDminus1} there exist two points $p_1$ and $p_2$ of respective multiplicity $m_1$ and $m_2$ with $m_1+m_2=d-1$. If $p_1,p_2$ are proper points in $\p^2$, then $\gamma\in \overline{\Bir(\p^2)_{d+1}}$ by Proposition~\ref{prop:mi+mj=d-k}. 

The set of elements $\gamma\in \Bir(\p^2)_d$ such that $p_1,p_2$ are proper being dense in each irreducible component of $\Bir(\p^2)_d$ (Proposition \ref{prop:BCM}), we obtain $\Bir(\p^2)_d\subset \overline{\Bir(\p^2)_{d+1}}$.

Similarly, if $\gamma$ has degree $d=8$, by Lemma~\ref{Lem:MultiplicitiesDminus1} there exist $p_1,p_2$ with $6=d-2 \leq m_1+m_2 \leq d-1=7$, hence we conclude as above that $\gamma$ is in either $\overline{\Bir(\p^2)_{10}}$ or $\overline{\Bir(\p^2)_{9}}$, the latter being included in $\overline{\Bir(\p^2)_{10}}$ by the first part of the proof.
\end{proof}

\begin{corollary}\label{Coro:kkover3}
Let $\gamma\in \Bir(\p^2)$ be a birational map of degree $d$. There exists an integer $k$ such that $1\le k\le \max\{1,\frac{d}{3}\}$ and
$\gamma\in \overline{\Bir(\p^2)_{d+k}}$.
\end{corollary}
\begin{proof}The proof is similar as the one of Corollary~\ref{Coro:ddp}.\end{proof}
\subsection{Degeneration associated to five general base-points}

Let us give another degeneration process, similar to Lemma~\ref{Lemm:QuadraticDeg} and Proposition~\ref{prop:mi+mj=d-k} but with more points. It will be useful to show that $\Bir(\p^2)_{10}\subset \overline{\Bir(\p^2)_{12}}$.

\begin{lemma}\label{Lemm:QuinticDeg}
Let $p_1,\dots,p_5$ be five distinct points of $\p^2$, such that no $3$ of them are collinear.

Then, there exists an open subset $U\subset \A^1$ containing $0$ and two morphisms $\nu\colon U\to \Bir(\p^2)$ and $p_6\colon U\to \p^2$, such that the following hold:

\begin{enumerate}[$(1)$]
\item
For $t\not=0$, the map $\nu(t)$ has degree $5$ and six base-points of multiplicity $2$, being $p_1,\dots,p_5,p_6(t)$, which are such that no $3$ are collinear and which do not belong to the same conic.\item
The map $\nu(0)$ is the identity and $p_6(0)$ belongs to to the conic passing through $p_1,\dots,p_5$.
\end{enumerate}
\end{lemma}
\begin{proof}Since no three of the points $p_1,\dots,p_5$ are collinear, we can assume that $p_1=[1:0:0]$, $p_2=[0:1:0]$, $p_3=[0:0:1]$.
We then denote by $\sigma$ the standard quadratic transformation $$\sigma\colon [x:y:z]\dashmapsto [yz:xz:xy].$$ Note that $\sigma$ is a local isomorphism at $p_4$, $p_5$ and that $p_1,p_2,p_3,\sigma(p_4),\sigma(p_5)$ are five points of $\p^2$ such that no $3$ are collinear.

Applying Lemma~\ref{Lemm:QuadraticDeg}, we obtain morphisms $\nu'\colon \A^1\to \Bir(\p^2)$ and $p_6'\colon \A^1\to \p^2$ such that $\nu'(0)$ is the identity and $\nu'(t)$ is a quadratic map with base-points $\sigma(p_4),\sigma(p_5),p_6'(t)$. Moreover, $p_6'(t)$ is collinear with $\sigma(p_4)$ and $\sigma(p_5)$ if and only if $t=0$. We can moreover choose that $p_6(0)$ does not belong to the triangle $xyz=0$, conjugate $\nu'$ with an automorphism of $\p^2$ if needed.

We denote by $U'\subset \A^1$ the dense open subset such that $p_6'(t)$ is not collinear with two of the points $p_1$, $p_2$, $p_3$, $\sigma(p_4)$,$\sigma(p_5)$ and does not belong to the conic through these points. In particular, $\nu'(t)$ is a local isomorphism at $p_1,p_2,p_3$ for each $t\in U'$. We have then a morphism map $\psi\colon U'\to  \PGL(3,\k)$  (or equivalently an element of $\PGL(3,\k(t))$) such that  $\psi(t)$ sends $\nu'(t)(p_i)$ onto $p_i$ for $i=1,2,3$. Since $\nu'(0)$ is the identity, we have $0\in U'$ and can choose $\psi(0)$ to be the identity. 

We then define a morphism $\nu\colon U'\to \Bir(\p^2)$ in the following way:
$$\nu(t)=\sigma\psi(t)\nu'(t)\sigma.$$
For $t=0$, $\nu(t)$ is the identity, since $\psi(0)$ and $\nu'(0)$ are the identity. It remains to observe that for a general $t\in U'$ the linear system of $\nu(t)$ has the desired form.

For $t\not=0$, the linear system of $\sigma\psi(t)$ consists of conics through $\nu'(t)(p_1)$, $\nu'(t)(p_2)$, $\nu'(t)(p_3)$. The linear system of 
$\sigma\psi(t)\nu'(t)$ consists  then of quartics having multiplicity two at $\sigma(p_4),\sigma(p_5),p_6'(t)$ and multiplicity one at $p_1,p_2,p_3$. The linear system of $\nu$ has then multiplicity $5$ and multiplicity $2$ at $p_1,p_2,p_3,p_4,p_5, \sigma(p_6'(t))$. We define $U\subset U'$ to be the union of $0$ with the points of $U'$ such that $p_6(t)=\sigma(p_6'(t))$ is a proper point of $\p^2$ and obtain the result. The degeneration of $\nu(t)$ comes because $p_6(0)$ belongs to the conic through $p_1,\dots,p_5$, which is the image by $\sigma$ of the line through $\sigma(p_4)$ and $\sigma(p_5)$.
\end{proof}
\begin{proposition}\label{prop:5points}
Suppose that $\gamma\in \Bir(\p^2)$ is a birational map of degree $d$ and has five proper base-points $p_1,\dots,p_5$ of multiplicity $m_1,\dots,m_5$, such that no $3$ of them are collinear and such that $\sum_{i=1}^5 m_i=2d-k$.

Then, there exists a morphism $\rho\colon U\to \Bir(\p^2)$, where $U\subset \A^1$ is an open subset containing $0$, such that $\rho(0)=\gamma$ and $\rho(t)$ has degree $d+2k$ for a general $t\not=0$.
\end{proposition}
\begin{proof}
We use the morphism $\nu\colon U\to \Bir(\p^2)$ given by Lemma~\ref{Lemm:QuinticDeg} and define $\rho$ as $\rho(t)=\gamma\circ \nu(t)^{-1}$. By construction, this is a morphism which satisfies $\rho(0)=\gamma$. Moreover, the degree of the map $\rho(t)$ for a general $t$ is equal to $5d-2m_1-2m_2-2m_3-2m_4-2m_5=d+2k$.
\end{proof}
\begin{remark}
If $p_1,\dots,p_r$ are the base-points of $\gamma$ of multiplicity $m_1,\dots,m_r$, the degeneration provided by Proposition~\ref{prop:5points} gives a family of birational maps of degree $d+2k$ with base-points of multiplicity $m_1+k,m_2+k,m_3+k,m_4+k,m_5+k,k,m_6,\dots,m_r$. The point of multiplicity $k$ created degenerates to a point which belongs to the conic through the first five points, and the linear system becomes the union of the linear system of $\gamma$ with $k$ times the conic.
\end{remark}
\begin{corollary}\label{Coro:101112}
We have  $\Bir(\p^2)_{10}\subset \overline{\Bir(\p^2)_{12}}$.
\end{corollary}
\begin{proof}
Each irreducible component of $\Bir(\p^2)_{10}$ corresponds to a $\Bir_\Lambda^\circ$ where $\Lambda=(d;m_1,\dots,m_k)$ is a proper homaloidal type (Proposition~\ref{prop:BCM} and Notation~\ref{Notation:BirLambda}). If there are two multiplicities $m_i,m_j$ such that $m_i+m_j=9$ or $m_i+m_j=8$, then Proposition~\ref{prop:mi+mj=d-k} shows that a general element of $\Bir_\Lambda^\circ$ belongs to $\overline{\Bir(\p^2)_{11}\cup \Bir(\p^2)_{12}}=\overline{\Bir(\p^2)_{12}}$, where the last equality follows from the fact that $\Bir(\p^2)_{11}\subset\overline{\Bir(\p^2)_{12}}$ by Corollary \ref{Coro:ddp}. Looking at Example~\ref{Exa:Types}, one sees that this holds for each proper homaloidal type of degree $10$, except for $\Lambda=(10;5^3,2^6)$. We then apply Proposition~\ref{prop:5points} to the first $5$ multiplicities, and obtain that a general element of $\Bir_\Lambda^\circ$ is contained in $\overline{\Bir(\p^2)_{12}}$.
\end{proof}

\subsection{Degeneration associated to five base-points, three of them being collinear}
We finish this section with another degeneration process, which works with maps having three collinear points. It will be useful to show that some maps of type $(8;4^3,2^3,1^3)$ belong to $\overline{\Bir(\p^2)_9}$ (Corollary~\ref{Cor8421limit}), although this is not true for a general element of type
$(8;4^3,2^3,1^3)$ (Corollary~\ref{Coro:GeneralLimitdplus1}), and the same for maps of type $(10;5^3,2^6)$.

\begin{lemma}\label{Lemm:QuarticDeg}
Let $p_1,\dots,p_5$ be five distinct points of $\p^2$, such that $p_1,p_2,p_3$ are collinear but no other triple of points belongs to the same line.

Then, there exists an open subset $U\subset \A^1$ containing $0$ and two morphisms $\nu\colon U\to \Bir(\p^2)$ and $p_6\colon U\to \p^2$, such that the following hold:

\begin{enumerate}[$(1)$]
\item
For $t\not=0$, the map $\nu(t)$ has degree $4$ and six base-points, namely $p_1,p_2,p_3$ with multiplicity $1$ and $p_4,p_5,p_6(t)$ with multiplicity $2$, and $p_6(t)$ is not collinear with any other base-point.\item
The map $\nu(0)$ is the identity and $p_6(0)$ belongs to to the line through $p_4$ and $p_5$.
\end{enumerate}
\end{lemma}
\begin{proof}
The points $p_1,p_2,p_4,p_5$ being in general position, we can assume, up to change of cooordinates, that
 $$p_1=[0:0:1],\ p_2=[1:1:1],\ p_4=[0:1:0],\ p_5=[1:0:0].$$
 This implies that $p_3=[1:1:a]$ for some $a\in \k^{*}$.

We consider the morphisms $\kappa,\rho,\tau\colon \A^1\to \Bir(\p^2)$ defined by
$$\begin{array}{rrcll}
\kappa(t)&\colon& [x:y:z]&\dashmapsto &[(ty-z)x: (tx-z)y: (tx-z)(ty-z)]\\
\rho(t)&\colon &[x:y:z]&\dashmapsto &[x(yt+z): yz: -z(yt+z)]\\
\tau(t)& \colon & [x:y:z] & \mapsto &[(a+t)y+z: (a-1)y: ax-(a+t)y]
\end{array}$$
(the map $\kappa$ is the same as in Lemma~\ref{Lemm:QuadraticDeg}). Observe that $\tau(t)\in \Aut(\p^2)$ for each $t$, that $\rho(0),\kappa(0)\in \Aut(\p^2)$ and that for a general $t$, $\kappa(t),\rho(t)$ are quadratic birational involutions of $\p^2$.
Moreover, the base-points of $\rho(t)$ are $p_4,p_5$ and the point infinitely to $p_5$ that corresponds to the line $yt+z=0$, that we will denote $p_6(t)$. The base-points of $\kappa(t)$ are $p_4,p_5,[1:1:t]$.

We then define $\nu\colon \A^1\dasharrow \Bir(\p^2)$ as $\nu(t)=\rho(0)\tau(0)^{-1}\kappa(0)\kappa(t)\tau(t)\rho(t)$. The linear system of $\nu(t)$ is given by comparing the linear system of $\kappa(t)\tau(t)$ with the one of $\rho(t)^{-1}=\rho(t)$. The linear system of $\kappa(t)\tau(t)$ consists of conics through $$\tau(t)^{-1}(\{p_4,p_5,[1:1:t]\})=\{[a+t: a: -a(a+t)],p_1,[t+1:1:-t-1]\}.$$
For $t\notin \{0,-1,a\}$, the three points are different from the three base-points of $\rho(t)$, so $\nu(t)$ has degree $4$, multiplicity $2$ at the three base-points of $\rho(t)$ and multiplicity $1$ at 
$$\rho(t)(\tau(t)^{-1})(\{p_4,p_5,[1:1:t]\})=\{p_2,p_1,p_3\}.$$
Choosing $U=\A^1\setminus \{-1,a\}$, we obtain the result.
\end{proof}

\begin{proposition}\label{prop:5points3collinear}
Suppose that $\gamma\in \Bir(\p^2)$ is a birational map of degree $d$ and has five proper base-points $p_1,\dots,p_5$ of multiplicity $m_1,\dots,m_5$, such that $p_1,p_2,p_3$ are collinear but no other triple of points belongs to the same line and such that $m_1+m_2+m_3+2m_4+2m_5=3d-k$.

Then, there exists a morphism $\rho\colon U\to \Bir(\p^2)$, where $U\subset \A^1$ is an open subset containing $0$, such that $\rho(0)=\gamma$ and $\rho(t)$ has degree $d+k$ for a general $t\not=0$.
\end{proposition}

\begin{proof}
We use the morphism $\nu\colon U\to \Bir(\p^2)$ given by Lemma~\ref{Lemm:QuarticDeg} and define $\rho$ as $\rho(t)=\gamma\circ \nu(t)^{-1}$. By construction, this is a morphism which satisfies $\rho(0)=\gamma$. Moreover, the degree of the map $\rho(t)$ for a general $t$ is equal to $4d-m_1-m_2-m_3-2m_4-2m_5=d+k$.
\end{proof}
\begin{corollary}\label{Cor8421limit}
If $\varphi\in \Bir(\p^2)$ is a map of type $(d;m_1,\dots,m_r)$ with five proper base-points $p_1,p_2,p_3,p_4,p_5$ of multiplicity $m_1,\dots,m_5$ respectively, such that $p_1,p_2,p_3$ are collinear but no other triple of points belongs to the same line. If $m_1+m_2+m_3=d-1$ and $m_4+m_5=d$ then $\varphi\in \overline{\Bir(\p^2)_{d+1}}$.

In particular, there exist some elements in $\overline{\Bir(\p^2)_9}$ of type $(8;4^3,2^3,1^3)$ and elements in $\overline{\Bir(\p^2)_{11}}$ of type $(10;5^3,2^6).$
\end{corollary}
\begin{proof}The second part directly follows from Proposition~\ref{prop:5points3collinear}. The second part follows by taking $(m_1,\dots,m_5)$ to be respectively $(1,2,4,4,4)$ and $(2,2,5,5,5)$.\end{proof}

\section{Restrictions on the degeneration in one degree less} \label{Soneless}

\begin{proposition}\label{Prop:Degen1}
Let $\varphi\in \Bir(\p^2)_d$ be a birational map of degree $d\ge 2$ with only proper base-points $($but $\varphi^{-1}$ can have infinitely near base-points$)$, and assume that $\varphi$ belongs to the closure of $\Bir(\p^2)_{d+1}$.

Then, there exist a set $\Omega$ consisting of one, two, three or four base-points of $\varphi$, which are collinear and such that the sum of their multiplicities is equal to $d-1$.
\end{proposition}
\begin{proof}
Suppose that $\varphi$ belongs to the closure of $\Bir(\p^2)_{d+1}$, which is equivalent to the fact that $\varphi^{-1}$ belongs to the closure of $\Bir(\p^2)_{d+1}$.

By Corollary~\ref{Coro:ClosureBird0}, there exist elements $\hat{f},\hat{g}\in {\overline{\Bir_{d+1}^{\circ}}}$ which are sent by $\pi_{d+1}$ onto $\varphi^{-1}$ and $\varphi$ respectively. Denoting by $f=[f_0:f_1:f_2]\in \Bir_d^{\circ}$ and $g=[g_0:g_1:g_2]\in \Bir_d^{\circ}$ the elements corresponding to $\varphi^{-1}$ and $\varphi$ respectively, there exists thus homogeneous polynomials $\alpha,\beta\in \k[x_0,x_1,x_2]$  of degree $1$, such that $\hat{f}=[\alpha f_0:\alpha f_1:\alpha f_2]$ and $\hat{g}=[\beta g_0:\beta g_1:\beta g_2]$ belong to $\overline{\Bir_{d+1}^{\circ}}$. 
By Proposition~\ref{Prop:BirLclosed}, there is a homaloidal type $\Lambda$ such that $\hat{f}\in \Bir_\Lambda$. 

Changing maybe $\beta$ and writing   $\Lambda^{*}=(d+1;n_1,\dots,n_r)$, this yields the existence (see Definition~\ref{Def:BirLambda}) of homogeneous polynomials $\hat{p}_1,\dots,\hat{p}_r$ of degree $n_1,\dots,n_r$ respectively, each of them contracted by $\hat{f}$ onto points $q_1,\dots,q_r\in \p^2$, being base-points of $\hat{g}$ of multiplicity at least $n_1,\dots,n_r$, and such that  $J(\hat{f})=\prod\limits_{i=1}^r \hat{p}_i$.

For each $i$, the fact that $\hat{p}_i$ is contracted by $\hat{f}=[\alpha f_0:\alpha f_1:\alpha f_2]$, and that $f=[f_0:f_1:f_2]$ is without common component imply that one of the following holds:
\begin{enumerate}
\item
$\hat{p}_i=\alpha$ and $q_i$ is any point of $\p^2$;
\item
$\hat{p}_i=\alpha p_i$, where $p_i$ is a polynomial contracted by $f$ onto $q_i$;
\item
$\hat{p}_i=p_i$, where $p_i$ is a polynomial contracted by $f$ onto $q_i$.
\end{enumerate}
Because $\varphi$ has only proper base-points, if a polynomial $p$ of degree $k$ is contracted by $f$ onto a point $q$, then $q$ is a base-point of $\varphi$, and thus of $g$, of multiplicity $k$.

The polynomials $p_i$ defined in $(2)$, $(3)$ above are thus irreductible factors of the Jacobian $J(f)$. Note that $\prod\limits_{i=1}^r \hat{p}_i=J(\hat{f})=J(f)\alpha^3$, so each polynomial contracted by $f$ appears exactly once in this decomposition, except if this polynomial is $\alpha$ itself.

$(a)$ We assume first that $\alpha$ is not a divisor of $J(f)$, which is the easiest case. We write $p_i=1$ in the case where $\hat{p}_i=\alpha$, and obtain then $\prod_{i=1}^r p_i=J(f)$. For each $i$, we denote by $m_i$ the degree of $p_i$, and obtain $m_i\in \{n_i,n_i-1\}$, and obtain, via Noether equalities
 $$\sum_{i=1}^r n_i=3d+6=3+ \sum_{i=1}^r m_i,$$
 $$\sum_{i=1}^r (n_i)^2=(d+1)^2-1=(2d+1)+d^2-1=(2d+1)+\sum_{i=1}^r (m_i)^2.$$
There are then exactly three values of $i$ such that $n_i=m_i+1$. Reordering such that the three indices are $1,2,3$, we obtain
$$2d+1=\sum_{i=1}^r (n_i)^2-\sum_{i=1}^r (m_i)^2=(2m_1+1)+(2m_2+1)+(2m_3+1)$$
which yields $m_1+m_2+m_3=d-1$. 

If $i\in \{1,2,3\}$ is such that $m_i>0$, then $\hat{p}_i=\alpha p_i$ and $p_i$ is contracted by $f$ onto $q_i$ (case $(2)$ above). Moreover, $q_i$ is a base-point of multiplicity $m_i$ of $g$ and of multiplicy $m_i+1$ of $\hat{g}$. This implies that $q_i$ belongs to the line given by $\beta=0$. Choosing $$\Omega=\{q_i \mid i\in \{1,2,3\} \mbox{ and }m_i>0\},$$ we obtain the result.

$(b)$ Assume now that $\alpha$ is a divisor of $J(f)$. Each other irreducible factor of $J(f)$ is then equal to exactly one $p_i$, and $\alpha$ appears four times in $J(f)\alpha^3=J(\hat{f})=\prod_{i=1}^r \hat{p}_i$. 

If $\hat{p}_i$ is equal to $\alpha$ or to $\alpha^2$ for some $i$, we will choose that $p_i=1$. In all other cases, the polynomial $p_i$ is defined as before. This implies that $J(f)=\alpha\prod_{i=1}^r p_i$. As before, we denote by $m_i$ the degree of $p_i$ and the Noether equalities yield $1+\sum_{i=1}^r m_i=3d-3$ and $1+\sum_{i=1}^r (m_i)^2=d^2-1.$

We consider now the following possibilities, which describe which one of the $\hat{p}_i$ are multiple of $\alpha$.

$(i)$ Suppose that $\alpha^2$ is equal to two different $\hat{p}_i$, that we can choose to be $\hat{p}_1$ and $\hat{p}_2$. We have then $n_1=n_2=2$, $m_1=m_2=0$ and $m_i=n_i$ for $i\ge 3$. Then $2d+1=\sum (n_i)^2-(\sum (m_i)^2+1)=3$, which is not possible since $d\ge 2$.

$(ii)$ Suppose that $\hat{p}_1=\alpha^2$ and that $\alpha$ divides two other $\hat{p}_i$, that we can assume to be  $\hat{p}_2$ and $\hat{p}_3$. We have then $n_1=2$, $m_1=0$, $n_2=m_2+1$, $n_3=m_2+1$ and $n_i=m_i$ for $i\ge 3$. Then $2d+1=\sum (n_i)^2-(\sum (m_i)^2+1)=1+(2m_1+1)+(2m_2+1)$, which yields $m_1+m_2=d-1$. We can conclude as before: if $m_i>0$ with $i\in \{2,3\}$, then $p_i$ is contracted onto $q_i$, which is a base-point of $\varphi$ of multiplicity $m_i$.

$(iii)$ The last case is when $\alpha^2$ does not divide any of the $\hat{p}_i$. There are thus exactly four values of $i$ such that $\alpha$ divides $\hat{p}_i$. We can choose that these are $1,2,3,4$, and obtain $\hat{p}_i=\alpha p_i$ for $i=1,2,3,4$, and $\hat{p}_i=p_i$ for $i>4$. So $n_i=m_i+1$ for $i\le 4$ and $n_i=m_i$ for $i>4$. In particular, we obtain $2d+1=\sum (n_i)^2-(\sum (m_i)^2+1)=(2m_1+1)+(2m_2+1)+(2m_2+1)+(2m_3+1)-1$, which yields $m_1+m_2+m_3+m_4=d-1$.
\end{proof}

\begin{corollary}\label{Coro:GeneralLimitdplus1}
Let $\varphi\in \Bir(\p^2)_d$ be a birational map of degree $d\ge 2$ with only proper base-points $($but $\varphi^{-1}$ can have infinitely near base-points$)$, such that no three are collinear.

Then, the following conditions are equivalent:

\begin{enumerate}[$(1)$]
\item
The map $\varphi$ belongs to the closure of $\Bir(\p^2)_{d+1}$.
\item
There exist a set $\Omega$ consisting of one or two base-points of $\varphi$ such that the sum of their multiplicities is equal to $d-1$.
\end{enumerate}
\end{corollary}
\begin{proof}
The implication $(1)\Rightarrow (2)$ is given by Proposition~\ref{Prop:Degen1}. The implication $(2)\Rightarrow (1)$ is given by Proposition~\ref{prop:mi+mj=d-k}.
\end{proof}

\begin{proposition} \label{d+1->d}
Let $\Lambda=(d;m_1,m_2,\ldots,m_r)$ be a proper homaloidal type.

The irreducible component $\pi_d(\Bir_\Lambda^\circ)$ of $\Bir(\PP^2)_d$ lies in the closure of $\Bir(\PP^2)_{d+1}$ if and only if there exists $m_i$ and $m_j$ such that $m_i+m_j=d-1$ or $m_i=d-1$.
\end{proposition}
\begin{proof}
The necessity of the condition on the multiplicities is given by Corollary~\ref{Coro:GeneralLimitdplus1}. 

Conversely, suppose that there exists $m_i$ and $m_j$ such that $m_i+m_j=d-1$ or   $m_i=d-1$. By Corollary~\ref{Coro:GeneralLimitdplus1},  a general element of $\pi_d(\Bir_\Lambda^\circ)$  is in the closure of $\Bir(\p^2)_{d+1}$. This gives the result.\end{proof}

\begin{remark}
The product of three general quadratic transformations is a general map of type $(8;4^3,2^3,1^3)$ that does not belong to $\overline{\Bir(\p^2)_9}$. However, there are some maps $\varphi\in \Bir(\p^2)$ of type $(8;4^3,2^3,1^3)$ that belong to $\overline{\Bir(\p^2)_9}$ (Corollary~\ref{Cor8421limit}). The same phenomenon occurs for $(10;5^3,2^6)$.
\end{remark}

Proposition~\ref{d+1->d} implies the following:

\begin{corollary}\label{Coro:Closureddp}
One has that $\Bir(\PP^2)_d\subset\overline{\Bir(\PP^2)_{d+1}}$ if and only if $d \in \{ 1,2,\ldots,6,7,9,11 \}$.
\end{corollary}

\begin{proof}
The ``if'' part follows from Corollary \ref{Coro:ddp}.
When $d=8$, there are exactly two proper homaloidal types, namely $(8;4^3,2^3,1^3)$ and $(8;3^7)$, cf.\ Table~\ref{tab:Hudson11}, whose corresponding irreducible components of $\Bir(\p^2)_8$ are not contained in  $\overline{\Bir(\p^2)_{9}}$ by Proposition~\ref{d+1->d}.
Similarly, when $d=10$, there are exactly 7 proper homaloidal types, cf.\ Table~\ref{tab:Hudson11}, whose corresponding irreducible components of $\Bir(\p^2)_{10}$ are not contained in  $\overline{\Bir(\p^2)_{11}}$.

It remains to see that $\Bir(\p^2)_d \not\subset \overline{\Bir(\p^2)_{d+1}}$ for each $d\geq12$. To do this, we use the proper homaloidal types $(3m;3m-6,6^{m-3},4^3,3^2,2,1)$ for $m\ge 4$, $(3m+1;3m-5,6^{m-2},4,3^3,1^4)$ for $m\ge 4$ and  $(3m+2;3m-4,6^{m-2},4^2,3^2,2^2,1 )$ for $m\ge 4$ (see Example~\ref{ex3m}). For each of these types of degree $d\in \{3m,3m+1,3m+2\}$, there are no two multiplicities $m_i$ and $m_j$ with $m_i+m_j=d-1$, hence  Proposition~\ref{d+1->d} says that the corresponding irreducible components of  $\Bir(\p^2)_d$ are not contained in $\overline{\Bir(\p^2)_{d+1}}$.
\end{proof}
We can now give the proof of Theorem~\ref{Thm:ClosureBirD}:
\begin{proof}[Proof of Theorem~$\ref{Thm:ClosureBirD}$]
It follows from Corollary~\ref{Coro:Closureddp} that $\overline{\Bir(\p^2)_{d}}=\Bir(\p^2)_{\le d}$ for $d\le 8$, that $\overline{\Bir(\p^2)_{d}}\not=\Bir(\p^2)_{\le d}$ for $d\in \{9,11\}$ and $d\ge 13$, and that $\Bir(\p^2)_{d-1}\subset \overline{\Bir(\p^2)_{d}}$ for $d\in \{10,12\}$.

The inclusions $\Bir(\p^2)_8\subset \overline{\Bir(\p^2)_{10}}$ and $\Bir(\p^2)_{10}\subset \overline{\Bir(\p^2)_{12}}$, given by Corollaries~\ref{Coro:ddp} and \ref{Coro:101112}, conclude the proof that $\Bir(\p^2)_{\le d}=\overline{\Bir(\p^2)_{d}}$ for each $d\in\{10,12\}$.
\end{proof}
\subsection{Examples}
The following example shows that it is possible that $[f_0h:f_1h:f_2h]\in \overline{\Bir_{d+m}^{\circ}}$ corresponds to a birational map $[f_0:f_1:f_2]\in \Bir_d^{\circ}$ that contracts the curve given by $h=0$.
\begin{example}
Let $\tilde{\kappa}\colon \A^1\to \Bir(\p^2)$ be given by 
$$\tilde{\kappa}(t)\colon [x:y:z]\dashmapsto [t(x^2-y^2)-xz: -yz: (t(x+y)-z)(t(x-y)-z)].$$
For each $t\not=0$, $\tilde{\kappa}(t)$ is a quadratic birational involution, whose three base-points are 
$$[1:- 1: 0], [1:1: 0], [1: 0: t],$$
and $\tilde\kappa(0)$ is the automorphism 
$$[x:y:z]\mapsto [-xz: -yz: z^2]=[-x:-y:z].$$
We now consider $\kappa\colon \A^1\to \Bir(\p^2)$ be the morphism that is given by $\kappa(t)=\tilde\kappa(t)\circ \sigma$, where $\sigma\colon [x:y:z]\dashmapsto [yz:xz:xy]$ is the standard quadratic involution of $\p^2$:
$$\kappa(t)\colon [x:y:z]\dashmapsto [(tz^2(y^2-x^2)-xy^2z): -x^2yz: (tz(x+y)-xy)(tz(y-x)-xy)].$$
For $t\not=0$, the linear system of the birational map $\kappa(t)$ consists of quartics having multiplicity $2$ at $[1:0:0]$, $[0:1:0]$, $[0:0:1]$, and having one tangent direction infinitely near to $[0:1:0]$ and the two tangent directions infinitely near to $[0:0:1]$.

For $t=0$, we obtain explicitely an element 
$$[-xy^2z: -x^2yz: x^2y^2]\in \overline{\Bir_4^{\circ}}\setminus \Bir_4^{\circ}$$
which corresponds to the birational map of degree $2$ given by
$$\kappa(0)\colon [x:y:z]\dashmapsto [-yz:-xz:xy].$$
The polynomial which multiplies this element of $\Bir_2^{\circ}$ to get an element of $\Bir_4$ is $xy$, and is here contracted by $\kappa(0)$.
\end{example}
The following example shows that one can obtain a general map of degree $2$ (no infinitely near base-points) as a limit of special maps of degree $3$ (having infinitely near base-points).
\begin{example}
Let $\sigma_1\in \Bir(\p^2)$ be the quadratic birational involution given by
$$[x:y:z]\dashmapsto [(x-y)(x-z): y(z-x), z(y-x)]$$
whose three base-points are $$[0:0:1],[0:1:0], [1:1:1].$$ Let $\sigma_2\colon\A^1\to \Bir(\p^2)$ be the morphism given by
$$\sigma_2(t)\colon [x:y:z]\dashmapsto [y(tx+z(1-t^2)):z(x-zt): y(x-zt)]$$
For each $t$, the map $\sigma_2(t)$ is a birational quadratic involution whose three base-points are
$$[1:0:0],[0:1:0],[t:0:1],$$
and $\sigma_2(0)$ is the standard quadratic transformation.

In particular, the morphism $\sigma_2\sigma_1\colon \A^1\to \Bir(\p^2)$ gives a degeneration of a family of cubic birational maps $\sigma_2\sigma_1(t)$ for $t\not=0$ to a quadratic map $\sigma_2\sigma_1(0)$ having only proper base-points. Moreover, for $t\not=0$ the fact that $[t:0:1]$, $[1:0:0]$ and $[0:0:1]$ are collinear implies that $\sigma_1\sigma_2(t)$ has one base-point infinitely near. In coordinates, we find 
$$\sigma_2(t)\sigma_1\colon [x:y:z]\dashmapsto [(x-z)y(tx+z(t^2-t-1)): (x-y)(x+z(t-1))z : (x-z)y(x+z(t-1))]$$
for $t=0$ we find an element 
$$[-(x-z)yz: (x-y)(x-z)z: (x-z)^2y]\in \overline{\Bir_3^{\circ}}\setminus \Bir_3^{\circ}$$
which corresponds to the birational map
$$[x:y:z]\dashmapsto [-yz:(x-y)z:(x-z)y]$$
having base-points at $[1:0:0]$, $[0:1:0]$, $[0:0:1]$.
\end{example}
\section{Halphen maps}\label{Sec:Halphen}
\subsection{Preliminaries on a family of Halphen homaloidal types}
Let us recall the notation of $\S\ref{SubSecHud}$:  we consider the free $\Z$-module $V$ of infinite countable rank, whose basis is $\{e_i\}_{i\in \mathbb{N}}$ and denote by $W$ the group of automorphisms of $V$  generated by $\sigma_0$ and by the permutations of the $e_i$ fixing $e_0$ (see $\S\ref{SubSecHud}$ for the definition of $\sigma_0$, which corresponds to the action of the standard quadratic transformation).
\begin{lemma}\label{BertiniHalphen}
$(1)$ 
The automorphism $B$ of $V$ that fixes $e_i$ for $i\ge 10$ and acts on $\bigoplus_{i=0}^9 \mathbb{Z}e_i$ via the matrix
$$\left(\begin{array}{rrrrrrrrrr}
17&0&6&6&6&6&6&6&6&6\\0&1&0&0&0&0&0&0&0&0\\-6&0&-3&-2&-2&-2&-2&-2&-2&-2\\-6&0&-2&-3&-2&-2&-2&-2&-2&-2\\-6&0&-2&-2&-3&-2&-2&-2&-2&-2\\-6&0&-2&-2&-2&-3&-2&-2&-2&-2\\-6&0&-2&-2&-2&-2&-3&-2&-2&-2\\-6&0&-2&-2&-2&-2&-2&-3&-2&-2\\-6&0&-2&-2&-2&-2&-2&-2&-3&-2\\-6&0&-2&-2&-2&-2&-2&-2&-2&-3\end{array}\right) $$
respectively to the basis $(e_0,\dots,e_9)$ belongs to the group $W$.

$(2)$ Denoting by $\nu\in W$ the transposition that exchanges $e_1$ and $e_2$, the matrix of $(\nu B)^{2a}\in W$ relative to $(e_0,\dots,e_9)$ is equal to
$$\left(\begin{array}{cccccccc}
36a^2& 12a^2-6a& 12a^2+6a& 12a^2& 12a^2& \dots & 12a^2\\
-12a^2-6a& -4a^2& -4a^2-4a& -4a^2-2a& -4a^2-2a& \dots & -4a^2-2a\\
-12a^2+6a& -4a^2+4a& -4a^2& -4a^2+2a& -4a^2+2a& \dots & -4a^2+2a\\
-12a^2& -4a^2+2a& -4a^2-2a& -4a^2& -4a^2& \dots & -4a^2& \\
-12a^2& -4a^2+2a& -4a^2-2a& -4a^2& -4a^2& \dots& -4a^2\\
\vdots & \vdots & \vdots & \vdots & \vdots& &  \vdots \\
-12a^2& -4a^2+2a& -4a^2-2a& -4a^2& -4a^2& \dots & -4a^2\\
\end{array}\right)+I $$
for each integer $a\in \mathbb{Z}$, where $I\in \mathrm{GL}(10,\mathbb{Z})$ is the identity matrix.
\end{lemma}
\begin{proof}
Assertion $(1)$ can be proven by hand, following the Hudson's test on the coefficients and applying then $\sigma_0$ and permutations. It can also be seen by observing that it is the action of a Bertini involution on the blow-up of $8$ general base-points.

Assertion $(2)$ is a straight-forward computation for $a=\pm 1$ and can be proved by induction on $\lvert a\rvert$ for the other integers.
\end{proof}
\begin{remark}
If we take nine points $p_1,\dots,p_9\in \p^2$ given by the intersection of two general cubics, the blow-up $X\to \p^2$ of these points gives a Halphen surface, whose anti-canonical morphism yields an elliptic fibration.

Moreover, the Bertini involutions (see \cite[$\S(1.3)$]{BayleBeauville}) associated to $8$ of the $9$ points lift to automorphisms of $X$ having actions on $\mathrm{Pic}(X)$ which are given by the first matrix of Lemma~\ref{BertiniHalphen}, up to permutation. The second matrix, for $a=1$ is then equal to the matrix of an automorphism $\tau\in X$. This implies that the matrix for $a\in \mathbb{Z}$ is the one given by~$\tau^a$.

See \cite{Giz80} for more details on the possible automorphisms of the Halphen surfaces.
\end{remark}
\begin{corollary}\label{Coro:HalphenType}
For each $a\ge 1$,
\begin{equation} \label{36a}
\Lambda_a=(36a^2+1; 12a^2+6a, 12a^2, 12a^2, 12a^2, 12a^2, 12a^2, 12a^2, 12a^2, 12a^2-6a)
\end{equation}
is a proper homaloidal type that satisfies $(\Lambda_a)^*=\Lambda_a$.
\end{corollary}
\begin{proof}
According to Proposition~\ref{Prop:HudsonW}, $\Lambda_a$ is proper if and only if it belongs to the orbit $W(e_0)$ of $e_0$ under the action of $e_0$.

It follows from Lemma~\ref{BertiniHalphen} that $\nu B\in W$, and that
$$(\nu B)^{2a}(e_0)=(36a^2+1)e_0-(12a^2+6a)e_1-(12a^2-6a)e_2-\sum_{i=3}^9 12a^2e_i,$$
which corresponds to the homaloidal type $\Lambda_a$.
Hence $\Lambda_a$ is a proper homaloidal type for each $a\ge 1$.  

Moreover, the homoloidal type $(\Lambda_a)^*$ is obtained by $(\nu B)^{-2a}(e_0)$ (Remark~\ref{Rem:MatrixInv}).
Using again Lemma~\ref{BertiniHalphen}, we obtain the equality  $$(\nu B)^{-2a}(e_0)=(36a^2+1)e_0-(12a^2-6a)e_1-(12a^2+6a)e_2-\sum_{i=3}^9 12a^2e_i,$$ which implies that $(\Lambda_a)^*=\Lambda_a$.
\end{proof}

The sequence of proper homaloidal types of Example~\ref{ex3m} suffices to show that for $d$ large, there are elements in $\Bir(\p^2)_d\setminus \overline{\Bir(\p^2)_{d+1}}$. However, all such families belong in fact to $\overline{\Bir(\p^2)_{d+2}}$. The following family of examples will be sufficient to prove Theorem~\ref{Theorem:NoBoundOnK}.

\begin{proposition}\label{Prop:Halphen}
For each $a\ge 1$, there exists a birational map $\tau_a$ of degree $d=36a^2+1$, which is of type
 $$(36a^2+1; 12a^2+6a, 12a^2, 12a^2, 12a^2, 12a^2, 12a^2, 12a^2, 12a^2, 12a^2-6a),$$
and which contracts exactly $9$ irreducible curves, $7$ of degree $12a^2$, one of degree $12a^2+6a$ and one of degree $12a^2-6a$.

Moreover $\tau_a\in \Bir(\p^2)_{d}$ does not belong to $\overline{\Bir(\p^2)_{d+k}}$ if $1\le k\le a$.\end{proposition}
\begin{proof}
By Corollary~\ref{Coro:HalphenType}, the type given above is a proper homaloidal type which is self-dual. Hence, by Proposition~\ref{Prop:HudsonW} there is a birational map $\tau_a$ having this type and having only proper base-points. We can moreover assume that $(\tau_a)^{-1}$ also has only proper base-points. This implies that $\tau_a$ contracts exactly  $9$ irreducible curves, $7$ of degree $12a^2$, one of degree $12a^2+6a$ and one of degree $12a^2-6a$ (Lemma~\ref{Lemm:DecompJac}).

We write $d=36a^2+1$ and $f=[f_0:f_1:f_2]\in \Bir_d^{\circ}$ the element sent on $\tau_a$ by $\pi_d$ and  suppose that $\hat{f}=[\alpha f_0:\alpha f_1:\alpha f_2]\in \Bir_{d+k}$ belongs to the closure of $\Bir_{d+k}^{\circ}$, for some homogeneous polynomial $\alpha$ of degree $k$. Hence, $\hat{f}$ belongs to $\Bir_\Lambda$ for some homaloidal type $(d+k;m_1,\dots,m_r)$ (Proposition~\ref{Prop:BirLclosed}). There exist then polynomials $p_1,\dots,p_r$ of degree $m_1,\dots,m_r$ respectively, each of them contracted by $\hat{f}$ onto points $q_1,\dots,q_r$, being base-points of $\hat{g}=[g_0:g_1:g_2]$ of multiplicity at least $m_1,\dots,m_r$ and satisfying that  $J(\hat{f})=\prod\limits_{i=1}^r p_i$. Moreover, $\pi_{d+k}(\hat{g})^{-1}=\pi_{d+k}(\hat{f})$.

Denote by $l_1,\dots,l_9$ the irreducible polynomials contracted by $[f_0:f_1:f_2]$, of degree $n_1,\dots,n_9$ respectively, with $$n_1=12a^2+6a, n_2=\dots=n_8=12a^2, n_9=12a^2-6a.$$ We have then 
$$\prod\limits_{i=1}^r p_i=J(\hat{f})=\alpha^3 J([f_0:f_1:f_2])=\alpha^3\prod\limits_{i=1}^9 l_i.$$

The polynomial $\alpha$ having degree $k\le a<12a^2-6a$, it is not a multiple of $l_i$ for any $i$. This implies that $l_il_jP$ is not contracted by $\hat{f}$ for any $1\le i,j\le 9$ and any homogeneous polynomial $P\not=0$. We can then reorder the $p_i$ such that:
\begin{enumerate}
\item
for $i=1,\dots,9$, $l_i$ divides $p_i$ and $p_i$ divides $l_i\alpha$;
\item
for $i\ge 10$, $p_i$ divides $\alpha$.
\end{enumerate}
Writing $m_i=n_i+\epsilon_i$ for $i=1,\dots,9$ and $m_i=\epsilon_i$ for $i\ge 10$ we have then $0\le \epsilon_i\le k$ for each $i$.

Applying Noether inequalities we obtain
$$\begin{array}{rcl}3k&=&(3(d+k)-3)-(3d-3)\\
&=&\sum m_i-\sum n_i\\
&=&\sum \epsilon_i\vpb\\
(d+k)^2-1&=&\sum (m_i)^2\vpb\\
&=&\sum\limits_{i\le 9} (n_i+\epsilon_i)^2+\sum\limits_{i\ge 10} (\epsilon_i)^2\vpb\\
&=&d^2-1+\sum\limits_{i\le 9} (2n_i\epsilon_i)+\sum (\epsilon_i)^2\\
&=&d^2-1+24a^2\sum\limits_{i\le 9} \epsilon_i+12a(\epsilon_1-\epsilon_9)+\sum (\epsilon_i)^2\\
\end{array}$$
The difference of both sides of the equation yields then
$$\begin{array}{rcl}
0&=&k^2+2dk-24a^2\sum\limits_{i\le 9} \epsilon_i-12a(\epsilon_1-\epsilon_9)-\sum (\epsilon_i)^2\\
&=&k^2+2dk-24a^2(3k-\sum\limits_{i\ge 10} \epsilon_i)-12a(\epsilon_1-\epsilon_9)-\sum (\epsilon_i)^2\\
&=&k^2+2k(d-36a^2)+\sum\limits_{i\ge 10} \epsilon_i(24a^2-\epsilon_i)-12a(\epsilon_1-\epsilon_9)-\sum\limits_{i\le 9} (\epsilon_i)^2\\
&\ge & k^2+2k+\sum\limits_{i\ge 10} \epsilon_i(24a^2-\epsilon_i)-12ak-9k^2\\
&=& 2k(1-6a-4k)+\sum\limits_{i\ge 10} \epsilon_i(24a^2-\epsilon_i)\\
&\ge & 2k-20a^2+\sum\limits_{i\ge 10} \epsilon_i(24a^2-\epsilon_i).
\end{array}$$
If $\epsilon_j>0$ for some $j>9$, we find 
$$\sum\limits_{i\ge 10} \epsilon_i(24a^2-\epsilon_i)\ge 24a^2-k>20a^2-2k,$$
so the above inequality implies that $\epsilon_j=0$ for all $j\ge 10$, which means that $r=9$.

Note that $f=[f_0:f_1:f_2]$ contracts $l_1,\dots,l_9$ onto $q_1,\dots,q_9$, which are then the base-points of $\pi_d(f)^{-1}=(\tau_a)^{-1}$. This implies that $\hat{f}$ contracts $p_1,\dots,p_9$ onto $q_1,\dots,q_9$.

Moreover, the points $q_1,\dots,q_9$ are base-points of $\hat{g}=[g_0:g_1:g_2]$ of multiplicity at least $m_1,\dots,m_9$. As we can choose the $9$ points in general position and since $(d+k;m_1,\dots,m_9)$ is a proper homaloidal type, the linear system of curves of degree $d+k$ having multiplicity at least $m_i$ at $q_i$ has dimension $2$ and corresponds to a birational map of degree $d+k$. This implies that the linear system $\sum \lambda_i g_i$ is irreducible, which leads to a contradiction.
\end{proof}

We can now finish the text with the proof of Theorem~\ref{Theorem:NoBoundOnK}:
\begin{proof}[Proof of Theorem~$\ref{Theorem:NoBoundOnK}$]
The first part follows from Proposition~\ref{Prop:Halphen}, the second part follows from Corollary~\ref{Coro:kkover3}.
\end{proof}


\begin{thebibliography}{Dem70xx}

\bibitem[A-C02]{alberich}
\textsc{Maria Alberich-Carrami\~nana}:
\textit{Geometry of the plane Cremona maps},
Lecture Notes in Mathematics \textbf{1769}, Springer-Verlag, Berlin, 2002.

\bibitem[BB00]{BayleBeauville}
\textsc{Lionel Bayle, Arnaud Beauville}:
\textit{Birational involutions of P2}. Kodaira's issue. 
Asian J. Math. \textbf{4} (2000), no. 1, 11--17.


\bibitem[BCM13]{BCM}
\textsc{Cinzia Bisi, Alberto Calabri, Massimiliano Mella}:
\textit{On plane Cremona transformations of fixed degree},
J. Geom. Anal. \textbf{25} (2015),  no. 2, 1108--1131.


\bibitem[Bla10]{Bl}
\textsc{J\'er\'emy Blanc}:
\textit{Groupes de Cremona, connexit\'e et simplicit\'e.}
Ann. Sci. Ec. Norm. Sup\'er. (4) {\bf 43} (2010), no. 2, 357-364.

\bibitem[BC13]{BC}
\textsc{J\'er\'emy Blanc, Serge Cantat}:
\textit{Dynamical Degrees of Birational transformations of projective surfaces},
to appear on J. Amer. Math. Soc., DOI: 10.1090/jams831.

\bibitem[BF13]{BF}
\textsc{J\'er\'emy Blanc, Jean-Philippe Furter}:
\textit{Topologies and structures of the Cremona groups},
Ann. of Math. \textbf{178} (2013), no. 3, 1173--1198.

\bibitem[Can11]{Can11}
\textsc{Serge Cantat}:
\textit{Sur les groupes de transformations birationnelles des surfaces},
 Ann. of Math. \textbf{174} (2011), no. 1, 299--340.

\bibitem[Can13]{Can13}
\textsc{Serge Cantat}:
\textit{Morphisms between Cremona groups and a characterization of rational varieties},
 Compos. Math. \textbf{150} (2014), no. 7, 1107--1124. 


\bibitem[Dem70]{De}  
\textsc{Michel Demazure}:
\textit{Sous-groupes alg\'ebriques de rang maximum du groupe de Cremona},
Ann. Sci. \'Ecole Norm. Sup. (4) {\bf 3} (1970), 507-588.

\bibitem[EF04]{EF04} 
\textsc{Eric Edo, Jean-Philippe Furter}:
\textit{Some families of polynomial automorphisms}, J. of Pure and Applied Algebra {\bf 194} (2004), 263-271.

\bibitem[EF04]{EL13} 
\textsc{Eric Edo, Drew Lewis}:
\textit{Some families of polynomial automorphisms III}, J. of Pure and Applied Algebra {\bf 219} (2015), no. 4, 864--874.

\bibitem[Fur97]{Fur97}
\textsc{Jean-Philippe Furter}:
\textit{On the variety of automorphisms of the affine plane},
 J. Algebra {\bf 195} (1997), 604-623.
 
\bibitem[Fur02]{Fur02}
\textsc{Jean-Philippe Furter}:
\textit{On the length of automorphisms of the affine plane},
 Math. Ann. {\bf 322} (2002), 401-411.

\bibitem[Fur13]{Fur13}
\textsc{Jean-Philippe Furter}:
\textit{Polynomial composition rigidity and plane polynomial automorphisms},
 J. London Math. Soc. \textbf{91} (2015), no. 1, 180--202.

\bibitem[Giz80]{Giz80}
\textsc{Marat. H. Gizatullin}:
\textit{Rational G-surfaces.} Izv. Akad. Nauk SSSR Ser. Mat., {\bf 44}(1):110--144, 239, 1980.


\bibitem[Hud27]{Hudson}
\textsc{Hilda Hudson}:
\textit{Cremona Transformations in Plane and Space},
Cambridge University Press, Cambridge, 1927.

\bibitem[Pop13]{Popov}
\textsc{Vladimir L. Popov}:
\textit{Tori in the Cremona groups}.
Izv. Math. {\bf 77} (2013), no. 4, 742--771. 

\bibitem[PR13]{PanRit}
\textsc{Ivan Pan, Alvaro Rittatore}:
\textit{Some remarks about the Zariski topology of the Cremona group.}
http://arxiv.org/abs/1212.5698.

\bibitem[Ser10]{Se}
\textsc{Jean-Pierre Serre}:
\textit{Le groupe de Cremona et ses sous-groupes finis.} 
S\'eminaire Bourbaki. Volume 2008/2009. Ast\'erisque No. {\bf 332} (2010), Exp. No. 1000, vii, 75--100.

\bibitem[Sha67]{Sha}
\textsc{Igor R. Shafarevich}:
\textit{Algebraic surfaces.} 
Proc. Steklov Inst. Math. {\bf 75}, 1967.

\end{thebibliography}
\end{document}